\documentclass[a4paper]{amsart}
\usepackage{graphicx}
\usepackage{amsmath}
\usepackage{amssymb}
\usepackage{oldgerm}
\usepackage[all]{xy}

\renewcommand{\mod}{\operatorname{mod}\nolimits}

\newcommand{\Ann}{\operatorname{Ann}\nolimits}

\newcommand{\Ext}{\operatorname{Ext}\nolimits}

\newcommand{\MaxSpec}{\operatorname{MaxSpec}\nolimits}

\newcommand{\rad}{\operatorname{rad}\nolimits}
\newcommand{\gr}{\operatorname{gr}\nolimits}

\newcommand{\m}{\operatorname{\mathfrak{m}}\nolimits}
\newcommand{\az}{\operatorname{\mathfrak{a}}\nolimits}

\newcommand{\La}{\Lambda}
\newcommand{\Z}{\operatorname{Z}\nolimits}
\newcommand{\cx}{\operatorname{cx}\nolimits}
\newcommand{\px}{\operatorname{px}\nolimits}

\newcommand{\V}{\operatorname{V}\nolimits}

\newcommand{\End}{\operatorname{End}\nolimits}

\newtheorem{theorem}{Theorem}[section]

\newtheorem{corollary}[theorem]{Corollary}

\newtheorem{lemma}[theorem]{Lemma}
\newtheorem{proposition}[theorem]{Proposition}
\theoremstyle{definition}
\newtheorem*{definition}{Definition}
\theoremstyle{definition}

\theoremstyle{definition}

\theoremstyle{definition}
\newtheorem*{example}{Example}
\theoremstyle{definition}

\theoremstyle{definition}

\theoremstyle{remark}

\theoremstyle{definition}

\theoremstyle{definition}
\newtheorem*{assumption}{Assumption}

\begin{document}
\title{Relative support varieties}
\author{Petter Andreas Bergh \& {\O}yvind Solberg}
\address{Institutt for matematiske fag \\ NTNU \\ N-7491 Trondheim
\\ Norway}
\email{bergh@math.ntnu.no} \email{oyvinso@math.ntnu.no}



\thanks{The first author was supported by NFR Storforsk grant no.\
167130}

\maketitle

\begin{abstract}
We define relative support varieties with respect to some fixed
module over a finite dimensional algebra. These varieties share many
of the standard properties of classical support varieties. Moreover,
when introducing finite generation conditions on cohomology, we show
that relative support varieties contain homological information on
the modules involved. As an application, we provide a new criterion
for a selfinjective algebra to be of wild representation type.
\end{abstract}

\section{Introduction}

Support varieties for modules over a given algebra are defined in
terms of the maximal ideal spectrum of some commutative graded ring
of cohomology operators, operators which act centrally on the
cohomology groups of the algebra. For group algebras of finite
groups, or, more generally, for finite dimensional cocommutative
Hopf algebras, this role is played by the cohomology ring of the
algebra (cf.\ \cite{Benson}, \cite{Carlson}, \cite{Evens},
\cite{Friedlander}). For commutative local complete intersections,
one uses the polynomial ring of Eisenbud operators (cf.\
\cite{Avramov1}, \cite{Avramov2}). In all these cases, the ring of
cohomology operators is Noetherian, and all the cohomology groups of
the algebra are finitely generated as modules. Consequently, the
theory of support varieties over these rings is very powerful, in
that the variety of a module contains a lot of homological
information on the module itself.

As shown in \cite{Snashall}, for a finite dimensional algebra, the
Hochschild cohomology ring, with its maximal ideal spectrum, is a
natural candidate as a ring of central cohomolgy operators. However,
in general this ring is not Noetherian, and the cohomology groups of
the algebra are not always finitely generated modules. But, as shown
in \cite{Erdmann}, when the Hochschild cohomology ring \emph{is}
Noetherian and all the cohomology groups are finitely genrated, then
one obtains a support variety theory very much like in the classical
cases. It is therefore important to establish which finite
dimensional algebras have ``nice" Hochschild cohomology rings. For
quantum complete intersections, this has been solved (cf.\
\cite{ErdmannSolberg}, \cite{BerghOppermann}).

In this paper, we define \emph{relative} support varieties with
respect to a fixed module. These are defined in terms of the maximal
ideal spectrum of some commutative graded subalgebra of the
$\Ext$-algebra of the module. As one would expect, these varieties
share many of the same properties of ``ordinary" support varieties,
such as the standard behavior on exact sequences etc. Moreover, when
we introduce finite generation conditions, then the relative support
varieties contain homological information on the modules involved,
just as in the classical case.

As an application, we provide a new criterion for a finite
dimensional selfinjective algebra to be of wild representation type.
Namely, we show that if there exists a module whose $\Ext$-algebra
is ``large" enough, Noetherian and finitely generated as a module
over its center, then the algebra is wild. This generalizes
Farnsteiner's theorem, which states that the complexity of every
module of a tame block of a finite group scheme is at most two (cf.\
\cite{Farnsteiner}).

\section{Relative support varieties}

Throughout this paper, we let $k$ be a field and $\La$ a finite
dimensional $k$-algebra. We denote by $\mod \La$ the category of
finitely generated left $\La$-modules, and we fix a module $M \in
\mod \La$ whose higher self-extensions do not all vanish. Whenever
we deal with $\La$-modules, we assume they belong to $\mod \La$.
Finally, for two $\La$-modules $X$ and $Y$, we denote by
$\Ext_{\La}^*(X,Y)$ the direct sum $\oplus_{i=0}^{\infty}
\Ext_{\La}^i(X,Y)$.

\sloppy Consider the $\Ext$-algebra $\Ext_{\La}^*(M,M)$ of $M$, in
which multiplication is given by the Yoneda product. Then for any
$\La$-module $N$, the graded $k$-vector space $\Ext_{\La}^*(M,N)$
is a graded right $\Ext_{\La}^*(M,M)$-module, whereas
$\Ext_{\La}^*(N,M)$ is a graded left $\Ext_{\La}^*(M,M)$-module.
Moreover, a $\La$-homomorphism $N_1 \xrightarrow{f} N_2$ induces
homomorphisms
\begin{eqnarray*}
\Ext_{\La}^*(M,N_1) & \xrightarrow{f_*} & \Ext_{\La}^*(M,N_2) \\
\Ext_{\La}^*(N_2,M) & \xrightarrow{f^*} & \Ext_{\La}^*(N_1,M)
\end{eqnarray*}
of right and left $\Ext_{\La}^*(M,M)$-modules. The homomorphism
$f_*$ is given as follows: given a homogeneous element
$$\eta \colon 0 \to N_1 \to X_n \to \cdots \to X_1 \to M \to 0$$
in $\Ext_{\La}^*(M,N_1)$, the element $f_* ( \eta )$ is the lower
exact sequence in the diagram
$$\xymatrix{
0 \ar[r] & N_1 \ar[d]^f \ar[r] & X_n \ar[d] \ar[r] & \cdots \ar[r] &
X_1 \ar[d] \ar[r] & M \ar[d] \ar[r] & 0 \\
0 \ar[r] & N_2 \ar[r] & K \ar[r] & \cdots \ar[r] & X_1 \ar[r] & M
\ar[r] & 0 }$$ in which the module $K$ is a pushout. Similarly,
the homomorphism $f^*$ is induced by pullback along $f$. It
follows immediately that $f_*$ and $f^*$ are well defined
homomorphisms of right and left $\Ext_{\La}^*(M,M)$-modules,
respectively.

The relative support varieties are defined with respect to some
commutative graded subalgebra of $\Ext_{\La}^*(M,M)$, and
therefore we now fix such a subalgebra.

\begin{assumption}
Fix a commutative graded subalgebra $H \subseteq
\Ext_{\La}^*(M,M)$ such that $H_0$ is a local ring.
\end{assumption}

As mentioned, the relative support varieties to be defined are
defined with respect to this graded subalgebra $H$. The assumption
that $H_0$ is local is made in order to get a nice
characterization of the trivial varieties. This assumption is not
very restrictive. For example, when we introduce finiteness
assumptions later, then we may actually take $H$ to be a
polynomial ring over $k$, so that $H_0$ is just $k$ itself.
Moreover, the following result shows that when $M$ is an
indecomposable module, then $H_0$ is automatically a local ring.

\begin{lemma}\label{local}
If $M$ is indecomposable, then $H_0$ is a local ring.
\end{lemma}

\begin{proof}
Since $H_0$ is a finite dimensional commutative $k$-algebra, the
factor algebra $H_0 / \rad H_0$ is a product $K_1 \times \cdots
\times K_t$ of fields. If $t \ge 2$, then this factor algebra
contains nontrivial idempotents, and these lift to $H_0$. But
$H_0$, being a subalgebra of $\End_{\La} (M)$, cannot contain any
nontrivial idempotent, hence $t=1$. Therefore the radical of $H_0$
is a maximal ideal.
\end{proof}

Since we have assumed that $H_0$ is a local ring, the graded ideal
$\rad H_0 \oplus H_1 \oplus \cdots$ is maximal in $H$ (and it is
the only maximal graded ideal). We denote this ideal by
$\m_{\gr}(H)$.

We now define a relative support variety theory for $\La$-modules,
in which the commutative graded ring $H$ is the coordinate ring.
Given a $\La$-module $N$, denote by $\Ann^i_H N$ the annihilator of
$\Ext_{\La}^*(M,N)$ in $H$, and by $\Ann^p_H N$ the annihilator of
$\Ext_{\La}^*(N,M)$. As the annihilator of any graded module over
any graded ring is graded, the ideals $\Ann^i_H N$ and $\Ann^p_H N$
are graded ideals of $H$. We define the injective and projective
\emph{support varieties of $N$ with respect to $H$} as
\begin{eqnarray*}
\V^i_H (N) & \stackrel{\text{def}}{=} & \{ \m \in \MaxSpec H
\mid \Ann^i_H N \subseteq \m \}, \\
\V^p_H (N) & \stackrel{\text{def}}{=} & \{ \m \in \MaxSpec H \mid
\Ann^p_H N \subseteq \m \},
\end{eqnarray*}
respectively, where $\MaxSpec H$ denotes the set of maximal ideals
of $H$. Note that $\Ann^i_H N$ and $\Ann^p_H N$ are contained in
$\m_{\gr}(H)$, hence $\m_{\gr}(H)$ is trivially a point in both
$\V^i_H (N)$ and $\V^p_H (N)$. We call a variety \emph{trivial} if
it only contains this point.

In the following result we record some elementary facts on relative
varieties. Whenever we write $\V_H^*(N)$ or $\Ann^*_H N$ and make a
statement, it is to be understood that the statement holds in both
the injective and projective cases. Furthermore, denote by
$M^{\perp}$ the category of all $\La$-modules $X$ such that
$\Ext_{\La}^n(M,X)=0$ for $n \gg 0$, and by $^{\perp}M$ the category
of all $\La$-modules $Y$ such that $\Ext_{\La}^n(Y,M)=0$ for $n \gg
0$.

\begin{proposition}\label{elementary}
For $\La$-modules $M,N,N_1,N_2,N_3$, the following hold:
\begin{itemize}
\item[(i)] $\V^*_H (M) = \MaxSpec H$. \item[(ii)] If
$N \in M^{\perp}$, then $\V^i_H (N)$ is trivial. In particular, this
holds if the injective dimension of $N$ is finite. \item[(iii)] If
$N \in {^{\perp}M}$, then $\V^p_H (N)$ is trivial. In particular,
this holds if the projective dimension of $N$ is finite.
\item[(iv)] For any exact sequence
$$0 \to N_1 \to N_2 \to N_3 \to 0,$$
the inclusion $\V^*_H (N_u) \subseteq \V^*_H (N_v) \cup \V^*_H
(N_w)$ holds whenever $\{ u,v,w \} = \{ 1,2,3 \}$. \item[(v)] If $N
= N_1 \oplus N_2$, then $\V^*_H (N) = \V^*_H (N_1) \cup \V^*_H
(N_2)$.
\end{itemize}
\end{proposition}

\begin{proof}
Since $H$ is a subalgebra of $\Ext_{\La}^*(M,M)$, no nonzero element
of $H$ can annihilate $\Ext_{\La}^*(M,M)$. Therefore $\Ann^*_H M
=0$, and this shows (i).

To prove (ii), note that if $\Ext_{\La}^n(M,N)=0$ for $n \gg 0$ and
$\eta$ is a homogeneous element in $H$ of positive degree, then some
power of $\eta$ belongs to $\Ann^i_H N$. Moreover, if $\theta$ is
any element of $\rad H_0$, then it is nilpotent, and therefore some
power of $\theta$ also belongs to $\Ann^i_H N$. Consequently
$\V^i_H(N) = \{ \m_{\gr}(H) \}$. This proves (ii), and the proof of
(iii) is similar.

As for (iv), we prove only the inclusion $\V^i_H (N_3) \subseteq
\V^i_H (N_1) \cup \V^i_H (N_2)$; the other inclusions are proved
analogously. The given short exact sequence induces an exact
sequence
$$\Ext_{\La}^*(M,N_2) \to \Ext_{\La}^*(M,N_3) \to
\Ext_{\La}^{*+1}(M,N_1)$$ of right $\Ext_{\La}^*(M,M)$-modules, from
which we obtain $\Ann_H^i N_2 \cdot \Ann_H^i N_1 \subseteq \Ann_H^i
N_3$. The inclusion $\V^i_H (N_3) \subseteq \V^i_H (N_1) \cup \V^i_H
(N_2)$ now follows.

The proof of (v) is straightforward.
\end{proof}

By combining properties (ii), (iii) and (iv) in Proposition
\ref{elementary}, we see that injective varieties are invariant
under cosyzygies, whereas projective varieties are invariant under
syzygies. We record these facts in the following slightly more
general result, which concludes this section.

\begin{corollary}\label{invariant}
Let $N$ be a $\La$-module, and let
$$0 \to N_1 \to N_2 \to N_3 \to 0$$
be an exact sequence in $\mod \La$.
\begin{itemize}
\item[(i)] If $N_2 \in M^{\perp}$, then $\V_M^i(N_1) =
\V_M^i(N_3)$. In particular, the injective variety of $M$ equals
that of $\Omega_{\La}^{-1}(M)$. \item[(ii)] If $N_2 \in
{^{\perp}M}$, then $\V_M^p(N_1) = \V_M^p(N_3)$. In particular, the
projective variety of $M$ equals that of $\Omega_{\La}^1(M)$.
\end{itemize}
\end{corollary}

\section{Finite generation}

The reason why the theories of support varieties for group rings,
cocommutative Hopf algebras and complete intersections are all
very powerful, is the existence of a central commutative
Noetherian ring over which all the cohomology groups are finitely
generated (cf.\ \cite{Avramov1}, \cite{Avramov2}, \cite{Benson},
\cite{Carlson}, \cite{Evens}, \cite{Friedlander}). As shown in
\cite{Erdmann}, a similar theory is obtained for support varieties
defined in terms of the Hochschild cohomology ring, when one
\emph{assumes} the existence of such a commutative ring. Motivated
by this, we now make the following assumption on the fixed
subalgebra $H$ of $\Ext_{\La}^*(M,M)$.

\begin{assumption}
The ring $H$ is Noetherian.
\end{assumption}

A priori, the algebra $H$ is just some \emph{unknown} graded
subalgebra of $\Ext_{\La}^*(M,M)$, and this is of course not
satisfactory if we want to do real computations. However, the
following result shows that when $H$ is a subalgebra of the center
$\Z (M)$ of $\Ext_{\La}^*(M,M)$, and we require
$\Ext_{\La}^*(M,M)$ to be a finitely generated $H$-module, then we
may take $H$ to be $\Z (M)$ itself. Note that $\Z (M)$ is a graded
algebra. Indeed, suppose $\eta$ is an element of $\Z (M)$, and
write $\eta = \eta_0 + \cdots + \eta_n$, where $\eta_i$ is an
element of $\Ext_{\La}^i(M,M)$ for each $i$. Let $\theta$ be any
homogeneous element of $\Ext_{\La}^*(M,M)$. Then since $\eta
\theta = \theta \eta$, we see that each $\eta_i$ must commute with
$\theta$. Therefore each $\eta_i$ belongs to $\Z (M)$, and this
shows that $\Z (M)$ is a graded algebra.

\begin{proposition}\label{choice}
The following are equivalent.
\begin{itemize}
\item[(i)] There exists a commutative Noetherian graded subalgebra
$R \subseteq \Z (M)$ over which $\Ext_{\La}^*(M,M)$ is a finitely
generated module. \item[(ii)] The ring $\Z (M)$ is Noetherian, and
$\Ext_{\La}^*(M,M)$ is a finitely generated $\Z (M)$-module.
\item[(iii)] The ring $\Ext_{\La}^*(M,M)$ is Noetherian and a finitely
generated $\Z (M)$-module.
\end{itemize}
\end{proposition}

\begin{proof}
The implication (ii) $\Rightarrow$ (i) is obvious. Suppose (i)
holds, and let $G$ be an algebra ``lying between" $R$ and
$\Ext_{\La}^*(M,M)$, i.e.\ $R \subseteq G \subseteq
\Ext_{\La}^*(M,M)$. Then $\Ext_{\La}^*(M,M)$ must be a finitely
generated $G$-module. Moreover, since $R$ is Noetherian and
$\Ext_{\La}^*(M,M)$ is a finitely generated $R$-module, we see that
$\Ext_{\La}^*(M,M)$ is a Noetherian ring. This shows the implication
(i) $\Rightarrow$ (iii). Finally, the implication (iii)
$\Rightarrow$ (ii) is \cite[Theorem 1]{Artin}.
\end{proof}

As mentioned in the previous section, the assumption that $H_0$ be
a local ring is superfluous once we have introduced finiteness
conditions. Namely, the following result shows that we may take
$H$ to be a polynomial ring over $k$, so that $H_0$ is just $k$
itself. Recall first that if $V$ is a graded $k$-vector space of
finite type (i.e.\ $\dim_k V_i < \infty$ for all $i$), then the
\emph{rate of growth} of $V$, denoted $\gamma (V)$, is defined as
$$\gamma (V) \stackrel{\text{def}}{=} \inf \{ t \in \mathbb{N} \cup \{
0 \} \mid \exists a \in \mathbb{R} \text{ such that } \dim_k V_n
\leq an^{t-1} \text{ for } n \gg 0 \}.$$

\begin{proposition}\label{normalization}
Let $N$ be a $\La$-module, and suppose $\Ext_{\La}^*(M,N)$
(respectively, $\Ext_{\La}^*(N,M)$) is a finitely generated
$H$-module. Then there exists a polynomial ring $k[x_1, \dots,
x_c] \subseteq H$, with $c = \gamma (H)$, such that $H$ and
$\Ext_{\La}^*(M,N)$ (respectively, $\Ext_{\La}^*(N,M)$) are
finitely generated $k[x_1, \dots, x_c]$-modules.
\end{proposition}

\begin{proof}
Follows from the Noether normalization lemma.
\end{proof}

In the following result we characterize precisely when \emph{all}
the cohomology modules are finitely generated over $H$.

\begin{proposition}\label{finite}
Consider the following conditions.
\begin{itemize}
\item[(i)] For all $N \in \mod \La$, the $H$-module
$\Ext_{\La}^*(M,N)$ is finitely generated.
\item[(ii)] The $H$-module $\Ext_{\La}^*(M, \La / \rad \La )$ is
finitely generated.
\item[(iii)] For all $N \in \mod \La$, the $H$-module
$\Ext_{\La}^*(N,M)$ is finitely generated.
\item[(iv)] The $H$-module $\Ext_{\La}^*( \La / \rad \La, M )$ is
finitely generated.
\end{itemize}
Then the implications \emph{(i)} $\Leftrightarrow$ \emph{(ii)} and
\emph{(iii)} $\Leftrightarrow$ \emph{(iv)} hold.
\end{proposition}

\begin{proof}
We prove only the implication (ii) $\Rightarrow$ (i); the
implication (iv) $\Rightarrow$ (iii) is proved analogously. The
proof is by induction on the length $\ell (N)$ of a module $N$.
Since the $H$-module $\Ext_{\La}^*(M, \La / \rad \La )$ is finitely
generated, so is $\Ext_{\La}^*(M,S)$ for any simple $\La$-module
$S$. Now suppose $\ell (N) > 1$, and choose a nonzero proper
submodule $L$ of $N$. The exact sequence
$$0 \to L \to N \to N/L \to 0$$
induces an exact sequence
$$\Ext_{\La}^*(M,L) \to \Ext_{\La}^*(M,N) \to \Ext_{\La}^*(M,N/L)$$
of $H$-modules. By assumption, both the end terms are finitely
generated $H$-modules, hence so is the middle term since $H$ is
Noetherian.
\end{proof}

There are situations when finiteness always occurs, regardless of
the module $M$ we start with. Namely, when all the cohomology groups
of the algebra are finitely generated over a central ring of
cohomology operators, as in the following definition.

\begin{definition}
The algebra $\La$ satisfies {\bf{Fg}} if there exists a commutative
Noetherian graded $k$-algebra $R = \bigoplus_{i=0}^{\infty} R_i$ of
finite type (i.e.\ $\dim_k R_i < \infty$ for all $i$) satisfying the
following:
\begin{enumerate}
\item[(i)] For every $X \in \mod \La$ there is a graded ring homomorphism
$$\phi_X \colon R \to \Ext_{\La}^*(X,X).$$
\item[(ii)] For each pair $(X,Y)$ of finitely generated
$\La$-modules, the scalar actions from $R$ on $\Ext_{\La}^*(X,Y)$
via $\phi_X$ and $\phi_Y$ coincide, and $\Ext_{\La}^*(X,Y)$ is a
finitely generated $R$-module.
\end{enumerate}
\end{definition}

As mentioned, this holds if $\La$ is the group algebra of a finite
group, a cocommutative Hopf algebra, a finite dimensional
commutative complete intersection, or if the Hochschild cohomology
ring of $\La$ is suitably ``nice" (cf.\ \cite{Avramov1},
\cite{Avramov2}, \cite{Benson}, \cite{Carlson}, \cite{Evens},
\cite{Friedlander}, \cite{Erdmann}, \cite{ErdmannSolberg},
\cite{BerghOppermann}). Now suppose $\La$ satisfies {\bf{Fg}} with
respect to a graded ring $R$ as in the definition, and let $X$ be a
$\La$-module. Then $\phi_X (R)$ is a commutative Noetherian graded
subalgebra of the center of $\Ext_{\La}^*(X,X)$. Moreover, for any
$Y \in \mod \La$ both $\Ext_{\La}^*(X,Y)$ and $\Ext_{\La}^*(Y,X)$
are finitely generated $\phi_X (R)$-modules.

When $\La$ satisfies {\bf{Fg}}, we may also define support varieties
with respect to the ring of cohomology operators. Namely, let $R$ be
as in the definition. Given $\La$-modules $X$ and $Y$, we define
$$\V_R (X,Y) \stackrel{\text{def}}{=} \{ \m \in \MaxSpec R \mid
\Ann_R \Ext_{\La}^*(X,Y) \subseteq \m \}.$$ Is this variety
comparable to $\V_{\phi_X (R)}^i(Y)$ and $\V_{\phi_Y (R)}^p(X)$? The
following result shows that the three varieties $\V_{\phi_X
(R)}^i(Y), \V_{\phi_Y (R)}^p(X)$ and $\V_R (X,Y)$ are in fact
isomorphic.

\begin{proposition}\label{isovar}
Suppose $\La$ satisfies {\bf{Fg}} with respect to a graded ring $R$
as in the definition above, and let $X$ and $Y$ be $\La$-modules.
Then the varieties $\V_{\phi_X (R)}^i(Y), \V_{\phi_Y (R)}^p(X)$ and
$\V_R (X,Y)$ are isomorphic.
\end{proposition}

\begin{proof}
Let $\m$ be a maximal ideal in $R$. Since $\phi_X \left ( \Ann_R
\Ext_{\La}^*(X,Y) \right )$ equals $\Ann^i_{\phi_X (R)} Y$, we see
that $\Ann_R \Ext_{\La}^*(X,Y) \subseteq \m$ if and only if
$\Ann^i_{\phi_X (R)} Y \subseteq \phi_X ( \m )$. Therefore $\m$
belongs to $\V_R (X,Y)$ if and only if $\phi_X ( \m )$ belongs to
$\V_{\phi_X (R)}^i(Y)$, and this shows that the varieties $\V_R
(X,Y)$ and $\V_{\phi_X (R)}^i(Y)$ are isomorphic. Similarly the
varieties $\V_R (X,Y)$ and $\V_{\phi_Y (R)}^p(X)$ are isomorphic.
\end{proof}

As we saw above, when $\La$ satisfies {\bf{Fg}} then for
\emph{every} $\La$-module $M$ there exists a commutative Noetherian
graded subalgebra $H \subseteq \Ext_{\La}^*(M,M)$ over which
$\Ext_{\La}^*( \La / \rad \La, M )$ and $\Ext_{\La}^*(M, \La / \rad
\La )$ are finitely generated. However, the following example shows
that this may very well hold for a module even if the algebra does
not satisfy {\bf{Fg}}.

\begin{example}
Suppose $\La$ is selfinjective, and let $M$ be a nonzero periodic
$\La$-module, i.e.\ $\Omega_{\La}^p (M) \simeq M$ for some $p \ge
1$. Then the first part of the minimal projective resolution of $M$
is a $p$-fold extension
$$0 \to M \to P_{p-1} \to \cdots \to P_0 \to M \to 0.$$
Denote this extension by $\mu$, and consider the subalgebra $k [ \mu
]$ of $\Ext_{\La}^*(M,M)$. This subalgebra is a Noetherian ring over
which $\Ext_{\La}^*(M,M)$ is finitely generated as a module. In
fact, given any $\La$-module $N$, the $k [ \mu ]$-modules
$\Ext_{\La}^*(M,N)$ and $\Ext_{\La}^*(N,M)$ are finitely generated
(cf.\ \cite{Schulz1} and \cite{Schulz2} for a discussion of these
phenomena). As an example, consider the quantum exterior algebra
$$k \langle x,y \rangle / (x^2, xy-qyx, y^2),$$
where the element $q$ is a nonzero non-root of unity in $k$. Let $M$
be a two dimensional vector space with basis $\{ u,v \}$, say. By
defining
$$xu=0, \hspace{.3cm} xv=0, \hspace{.3cm} yu=v, \hspace{.3cm} yv=0,
\hspace{.3cm}$$ this vector space becomes a module over the quantum
exterior algebra. Moreover, it is not difficult to see that this
module is periodic of period one (cf.\ \cite[Example 4.5]{Bergh1}).
However, by \cite{ErdmannSolberg} and \cite[Theorem
5.5]{BerghOppermann} the algebra does not satisfy {\bf{Fg}}, since
$q$ is not a root of unity.
\end{example}

We now return to the general theory. A natural question to ask is
how big the relative support variety of a module is. For an
arbitrary module $N$, this cannot be answered unless we introduce
finiteness conditions, since a priori there is no relationship
between $H$ and $\Ext_{\La}^*(M,N)$ or $\Ext_{\La}^*(N,M)$.
However, when we introduce finite generation, the situation
becomes much more manageable.

Let $X$ be a $\La$-module with minimal projective and injective
resolutions
$$\cdots \to P_2 \to P_1 \to P_0 \to X \to 0,$$
$$0 \to X \to I^0 \to I^1 \to I^2 \to \cdots,$$
say. Then we define the \emph{complexity} and \emph{plexity} of $X$,
denoted $\cx X$ and $\px X$, respectively, as
\begin{eqnarray*}
\cx X & \stackrel{\text{def}}{=} & \inf \{ t \in \mathbb{N} \cup
\{ 0 \} \mid \exists a \in \mathbb{R} \text{ such that } \dim_k
P_n
\leq an^{t-1} \text{ for } n \gg 0 \}, \\
\px X & \stackrel{\text{def}}{=} & \inf \{ t \in \mathbb{N} \cup
\{ 0 \} \mid \exists a \in \mathbb{R} \text{ such that } \dim_k
I^n \leq an^{t-1} \text{ for } n \gg 0 \}.
\end{eqnarray*}
The complexity and the plexity of a module are not necessarily
finite. Also, from the definition we see that $\cx X =0$
(respectively, $\px X =0$) if and only if $X$ has finite
projective dimension (respectively, finite injective dimension).
It is well known that the complexity of $X$ equals $\gamma \left (
\Ext_{\La}^*(X, \La / \rad \La ) \right )$, whereas its plexity
equals $\gamma \left ( \Ext_{\La}^*( \La / \rad \La, X ) \right
)$. Generalizing this, we define the \emph{complexity of the pair}
($X,Y$) of $\La$-modules to be $\gamma \left ( \Ext_{\La}^*(X,Y)
\right )$, and denote it by $\cx (X,Y)$. Thus $\cx X$ is the
complexity of the pair $(X, \La / \rad \La )$, whereas $\px X$ is
the complexity of the pair $( \La / \rad \La, X )$. Note that $\cx
(X,Y) \neq \cx (Y,X)$ in general, that is, the order matters.
Also, it follows from the discussion prior to \cite[Proposition
5.3.5]{Benson} that $\cx (M,N) \le \cx M$, and similarly $\cx
(M,N) \le \px N$. In particular $\cx (M,M)$ is at most $\cx M$ and
$\px M$, and the following result shows that equality occurs when
finite generation holds.

\begin{proposition}\cite[Proposition 5.3.5]{Benson}\label{cx}
If the $H$-module $\Ext_{\La}^*(M, \La / \rad \La )$ is finitely
generated, then
$$\cx M = \cx (M,M) = \gamma (H).$$
Similarly, if the $H$-module $\Ext_{\La}^*( \La / \rad \La, M )$ is
finitely generated , then
$$\px M = \cx (M,M) = \gamma (H).$$
In particular, if both $\Ext_{\La}^*(M, \La / \rad \La )$ and
$\Ext_{\La}^*( \La / \rad \La, M )$ are finitely generated over $H$,
then $\cx M = \px M$.
\end{proposition}

As for the ``size" of the relative support varieties, the following
result shows that it is given in terms of the complexity, provided
finite generation holds.

\begin{proposition}\label{dim}
If the $H$-module $\Ext_{\La}^*(M,N)$ is finitely generated, then
$\dim \V_H^i (N) = \cx (M,N)$. Similarly, if the $H$-module
$\Ext_{\La}^*(N,M)$ is finitely generated , then $\dim \V_H^p (N) =
\cx (N,M)$.
\end{proposition}

\begin{proof}
If $\Ext_{\La}^*(M,N)$ is finitely generated over $H$, then $\gamma
\left ( H / \Ann_H^i N \right ) = \gamma \left ( \Ext_{\La}^*(M,N)
\right )$, and so by definition $\dim \V_H^i (N) = \cx (M,N)$. The
other equality is proved similarly.
\end{proof}

\section{Wild algebras and complexity}

In this section we assume that our field $k$ is algebraically
closed. Recall that $\La$ is of \emph{finite representation type} if
there are only finitely many non-isomorphic indecomposable
$\La$-modules. Furthermore, recall that $\La$ is of \emph{tame
representation type} if there exist infinitely many non-isomorphic
indecomposable $\La$-modules, but they all belong to one-parameter
families, and in each dimension there are finitely many such
families. Finally, the algebra $\La$ is of \emph{wild representation
type} if it is not of finite or tame type.

In \cite{Crawley}, Crawley-Boevey established a link between the
representation type of a selfinjective finite dimensional algebra
and the complexities of its modules. Namely, it was shown that for
such an algebra, in any dimension only finitely many indecomposable
modules are not of complexity one. Using this, Farnsteiner showed in
\cite{Farnsteiner} that the complexity of every module of a tame
block of a finite group scheme is at most two.

Suppose our algebra $\La$ is selfinjective and satisfies the
``global" finite generation hypothesis {\bf{Fg}} defined immediately
after Proposition \ref{finite}. That is, suppose there exists a
commutative Noetherian graded $k$-algebra $R =
\bigoplus_{i=0}^{\infty} R_i$ of finite type satisfying the
following:
\begin{enumerate}
\item[(i)] For every $X \in \mod \La$ there is a graded ring homomorphism
$$\phi_X \colon R \to \Ext_{\La}^*(X,X).$$
\item[(ii)] For each pair $(X,Y)$ of finitely generated
$\La$-modules, the scalar actions from $R$ on $\Ext_{\La}^*(X,Y)$
via $\phi_X$ and $\phi_Y$ coincide, and $\Ext_{\La}^*(X,Y)$ is a
finitely generated $R$-module.
\end{enumerate}
Then Farnsteiner's proof still applies, hence $\La$ is wild if there
exists a module of complexity at least three. We end this paper with
the following result, which generalizes this.

\begin{theorem}\label{cxwild1}
Suppose $\La$ is selfinjective, and there exists a $\La$-module $M$
satisfying the following:
\begin{itemize}
\item[(i)] $\cx (M,M) \ge 3$,
\item[(ii)] there exists a commutative
Noetherian graded subalgebra $H \subseteq \Ext_{\La}^*(M,M)$ over
which $\Ext_{\La}^*(M,M)$ is a finitely generated module.
\end{itemize}
Then $\La$ is of wild representation type.
\end{theorem}

\begin{proof}
Suppose (i) holds. By Proposition \ref{normalization}, we may assume
that $H$ is a polynomial ring, say $H = k[y_1, \dots, y_n]$, where
$n \ge 3$. Denote the ideal $\Ann_H^i N \subseteq H$ by $\az$. Since
$\Ext_{\La}^*(M,M)$ is a finitely generated $H / \az$-module, we may
apply the Noether normalization lemma and obtain a new polynomial
ring $R = k[x_1, \dots, x_c] \subseteq H/ \az$, with $c = \cx
(M,M)$, over which $\Ext_{\La}^*(M,M)$ is finitely generated.
Moreover, we may assume that the homogeneous elements $x_1, \dots,
x_c$ are of the same degree, say $| x_i | =d$.

The maximal ideals in $R$ correspond to points $( \alpha_1, \dots,
\alpha_c ) \in k^c$. Given an ideal $I \subseteq R$, we denote its
variety by $\V_R (I)$, thus
$$\V_R(I) = \{ \alpha \in k^c \mid f ( \alpha )=0 \text{ for all
} f \in I \}.$$ Now for each $\alpha \in k$, denote the element
$x_1+ \alpha x_2 \in R$ by $x_{\alpha}$. Lifting this element to $H$
gives a homogeneous element $\eta_{\alpha} \in \Ext_{\La}^*(M,M)$ of
degree $d$, from which we obtain a short exact sequence
\begin{equation*}\label{ES}
0 \to M \to K_{\alpha} \to \Omega_{\La}^{d-1}(M) \to 0.
\tag{$\dagger$}
\end{equation*}
By applying the same proof as in \cite[Proposition 4.3]{Erdmann}, we
see that $\V_R^p (K_{\alpha}) = \V_R ( x_1+ \alpha x_2 )$. Thus $\cx
(K_{\alpha},M) =c-1$, and $K_{\alpha}$ is not isomorphic to
$K_{\alpha'}$ whenever $\alpha \neq \alpha'$.

For each $\alpha$, let $K_{\alpha} = K_{\alpha}^1 \oplus \cdots
\oplus K_{\alpha}^{t_{\alpha}}$ be a decomposition of $K_{\alpha}$
into indecomposable $\La$-modules. Moreover, denote the ideal
$\Ann_R^p K_{\alpha}^i \subseteq R$, that is, the ideal $\Ann_R
\Ext_{\La}^* ( K_{\alpha}^i, M )$, by $\az_i$. Then $\V_R^p (
K_{\alpha}^i )$ is by definition the variety $\V_R ( \az_i )$, and
therefore
$$\V_R ( x_1+ \alpha x_2 ) = \V_R^p (K_{\alpha}) =
\bigcup_{i=1}^{t_{\alpha}} \V_R^p ( K_{\alpha}^i ) =
\bigcup_{i=1}^{t_{\alpha}} \V_R ( \az_i ) =  V_R (
\prod_{i=1}^{t_{\alpha}} \az_i ),$$ which in turn implies
$$\prod_{i=1}^{t_{\alpha}} \az_i \subseteq \sqrt{\prod_{i=1}^{t_{\alpha}}
\az_i} = \sqrt{(x_1+ \alpha x_2 )}.$$ Since the ideal $(x_1+ \alpha
x_2 )$ is prime, it is equal to its own radical, and contains one of
the ideals $\az_1, \dots, \az_{t_{\alpha}}$, say $\az_1$. However,
the variety $\V_R^p ( K_{\alpha}^1 )$ is contained in $\V_R^p (
K_{\alpha} )$, and therefore $(x_1+ \alpha x_2 ) = \sqrt{\az_1}$.
Consequently $\sqrt{\az_1} = (x_1+ \alpha x_2 )$, and this shows
that the varieties $\V_R^p ( K_{\alpha}^1 )$ and $\V_R^p (
K_{\alpha} )$ are equal.

The indecomposable $\La$-modules $\{ K_{\alpha}^1 \}_{\alpha \in k}$
are pairwise nonisomorphic, and from the exact sequence (\ref{ES})
we see that $\dim_k K_{\alpha}^1 \le \dim_k M + \dim_k
\Omega_{\La}^{d-1}(M)$ for every $\alpha \in k$. Moreover, by
construction we know that $\cx (K_{\alpha}^1,M) = \cx
(K_{\alpha},M)$, hence
$$2 \le c-1 = \cx (K_{\alpha}^1,M) \le \cx K_{\alpha}^1.$$
The result now follows from Crawley-Boevey's result \cite[Theorem
D]{Crawley}.
\end{proof}

Using Proposition \ref{finite} and Proposition \ref{cx}, we obtain
the following corollaries to Theorem \ref{cxwild1}.

\begin{corollary}\label{cxwild2}
Suppose $\La$ is selfinjective, and there exists a $\La$-module $M$
satisfying the following:
\begin{itemize}
\item[(i)] $\cx M \ge 3$,
\item[(ii)] there exists a commutative Noetherian graded subalgebra $H
\subseteq \Ext_{\La}^*(M,M)$ over which $\Ext_{\La}^*(M, \La / \rad
\La)$ is a finitely generated module.
\end{itemize}
Then $\La$ is of wild representation type.
\end{corollary}

\begin{corollary}\label{cxwild3}
Suppose $\La$ is selfinjective, and there exists a $\La$-module $M$
satisfying the following:
\begin{itemize}
\item[(i)] $\px M \ge 3$,
\item[(ii)] there exists a commutative Noetherian graded subalgebra $H
\subseteq \Ext_{\La}^*(M,M)$ over which $\Ext_{\La}^*( \La / \rad
\La, M)$ is a finitely generated module.
\end{itemize}
Then $\La$ is of wild representation type.
\end{corollary}

Finally, by applying Proposition \ref{choice} we obtain the
following corollary to Theorem \ref{cxwild1}. It shows that an
algebra is wild if it possesses a module whose $\Ext$-algebra is
``big enough", Noetherian and finitely generated over its center.

\begin{corollary}\label{cxwild3}
Suppose $\La$ is selfinjective, and there exists a $\La$-module $M$
satisfying the following:
\begin{itemize}
\item[(i)] $\cx (M,M) \ge 3$,
\item[(ii)] $\Ext_{\La}^*(M,M)$ is a Noetherian ring and finitely
generated as a module over its center.
\end{itemize}
Then $\La$ is of wild representation type.
\end{corollary}

We end this paper with the following example illustrating Theorem
\ref{cxwild1}, an example in which the algebra does not satisfy the
finite generation hypothesis {\bf{Fg}}.

\begin{example}
Let $q$ be a nonzero non-root of unity in $k$, and denote by
$\Gamma$ the quantum exterior algebra
$$k \langle x,y \rangle / (x^2, xy-qyx, y^2).$$
Let $X$ be the $\Gamma$-module from the example following
Proposition \ref{isovar}, i.e.\ $X$ is a two dimensional vector
space with basis $\{ u,v \}$, say, and with scalar multiplication
defined by
$$xu=0, \hspace{.3cm} xv=0, \hspace{.3cm} yu=v, \hspace{.3cm} yv=0.
\hspace{.3cm}$$ This module is periodic of period one, and so if we
denote its projective cover
$$0 \to X \to P \to X \to 0$$
by $\mu$, the $\Ext$-algebra $\Ext_{\Gamma}^*(X,X)$ is finitely
generated as a module over the polynomial subalgebra $k[ \mu ]$. Now
let $\La$ be the algebra $\Gamma \otimes_k \Gamma \otimes_k \Gamma$,
and let $M$ be the $\La$-module $X \otimes_k X \otimes_k X$. Then
the $\Ext$-algebra of $M$ is given by
$$\Ext_{\La}^*(M,M) = \Ext_{\Gamma}^*(X,X) \overline{\otimes}_k
\Ext_{\Gamma}^*(X,X) \overline{\otimes}_k \Ext_{\Gamma}^*(X,X),$$
where $\overline{\otimes}$ differ from the usual tensor product only
in that elements of odd degree anticommute (cf.\ \cite[Chapter
XI]{Cartan}). Therefore $\Ext_{\La}^*(M,M)$ is finitely generated as
a module over a commutative Noetherian graded subalgebra. Moreover,
since $\cx_{\Gamma} (X,X) =1$, we see that $\cx_{\La} (M,M)=3$.
Finally, the algebra $\La$ does not satisfy {\bf{Fg}}. Namely, this
algebra is a quantum exterior algebra on six generators, where some
of the defining commutators equal $q$. Then by \cite{ErdmannSolberg}
and \cite[Theorem 5.5]{BerghOppermann} $\La$ does not satisfy
{\bf{Fg}}, since $q$ is not a root of unity.
\end{example}

\end{document}